\documentclass[12pt]{amsart}
\usepackage{amsthm}
\usepackage{amssymb}
\usepackage{amsmath}
\usepackage{graphicx}
\usepackage{amscd}
\usepackage{amsfonts}
\usepackage{amsbsy}
\usepackage[T1]{fontenc}
\usepackage[english]{babel}
\usepackage{xcolor}
\usepackage{a4wide}
\usepackage{bbm}

\usepackage{epsfig}
\usepackage{amssymb}

\newtheorem{theorem}{Theorem}
\newtheorem{corollary}[theorem]{Corollary}
\newtheorem{remark}[theorem]{Remark}

\newtheorem{lemma}[theorem]{Lemma}

\newcommand{\Q}{\mathbb Q}
\newcommand{\N}{\mathbb N}
\newcommand{\T}{\mathbb T}

\newcommand{\zf}{\zeta_f}

\newcommand{\SSS}{{\mathbb S}}

\newcommand{\trace}{\text{trace}}

\author[M.J. Gonz\'alez]{Marcos J. Gonz\'alez}
\address{ Departamento de Matem\'aticas, Universidad Sim\'on
Bol\'{\i}var, Apartado 89000, Caracas 1086-A, Venezuela}
\email{mago@usb.ve, gmarcos76@gmail.com}

\author[V.F. Sirvent]{V\'{\i}ctor F. Sirvent}
\address{Departamento de Matem\'aticas, Universidad Cat\'olica del Norte, Antofagasta, Chile}
\email{victor.sirvent@ucn.cl}

\author[R. Urz\'ua]{Richard Urz\'ua}
\address{Departamento de Matem\'aticas, Universidad Cat\'olica del Norte, Antofagasta, Chile}
\email{rurzua@ucn.cl}

\title{Homological data on the periodic structure  of self-maps on  wedge sums}

\date{\today}



\begin{document}

\begin{abstract}
In this article, we study the periodic points for continuous self-maps on a wedge sum of topological manifolds, exhibiting a particular combinatorial structure.
We compute explicitly the Lefschetz numbers, the Dold coefficients and consider its  set of algebraic periods.
Moreover, we study the special case of maps on a wedge sum of tori, and show some of  the homological obstructions present in defining these  maps.
\end{abstract}

\maketitle

\section{Introduction}\label{s:introduction}

An important problem  common to  the theory of dynamical systems and fixed point theory is the study of periodic points and the period sets of a continuous self-map on a compact topological space. 
One of the approaches widely used is via the Lefschetz fixed point theorem and its refinements~(\emph{cf.}~\cite{BGW,JM}).

\smallskip

We are interested in finding certain homological information bounding the number of periodic points and their periods
for a continuous self-map on a wedge sum of topological manifolds. 
We recall that the wedge sums of those spaces are not topological manifolds~(\cite[p. 10]{hatcher}).
One class of  simple maps on this category are  graph maps, i.e. self-maps on a wedge sum of circles, whose periodic point set structure has been extensively studied, see for instance~\cite{ALLM,G-Ll:2,G-Ll-W,jiang,Ll-Mi,Ll-Sa,Ll-Si:2021}.
Also, self-maps on a wedge sum of spheres has been studied in different contexts~\cite{babo,BGW,Ll-Si:2013,MP,Si:2023}.
Besides these cases, the dynamics of maps on wedge sums of other spaces have not been studied in detail as far as we know.

\smallskip

In this article we deal with the calculation of the Lefschetz numbers, the Lefschetz zeta functions, Dold coefficients (the Dold coefficient of order $m$ is also known in the literature as the Lefschetz number of period $m$) and the algebraic periods for general self-maps on wedge sums of compact topological manifolds. We give explicit and closed formulae in terms of some more basic maps called  coordinate maps.
In particular,  Theorem~\ref{thm:main1} reduces the Lefschetz numbers and the Lefschetz zeta function of the map $f$ in terms of the Lefschetz numbers and zeta functions of its coordinate maps.
Theorem~\ref{thm:l-chica}  deals with the Dold coefficients and Corollary~\ref{coro:aperL} with the  algebraic periods for such map. 

\smallskip

The study of maps without periodic points has been done in different contexts and  spaces, see for instance~\cite{BW,Ha,Ja,Ko}. 
A weaker notion is the  concept of Lefschetz periodic point free map, i.e. a map whose all Lefschetz numbers are equal to zero, for  example see these articles in the subject~\cite{GKNS,G-Ll:2011,Ll-Si:2013,Si:2023}.
In Section~\ref{s:tori} we deal about this subject and we prove in 
Theorem~\ref{thm: torus-lppf}, that establishes that a continuous permutative self map on a wedge sum of tori is not Lefschetz periodic point free unless the induced maps on homology have zero eigenvalues.

\smallskip

We highlight  the importance of the study of the dynamics of self-maps on wedge-sums; since the computation of the Lefschetz numbers, Dold coefficients, algebraic periods and other information based on homology invariants of any map defined on a compact polyhedron (or more generally  an ENR) can be reduced to a map defined on a wedge sum of spheres (of different dimensions).

\smallskip

The article is structured as follows in Section~\ref{s:2} we introduce the basic notions and present the main results and statements of the article.
In Section~\ref{s:proofs}, we prove the main results and introduce the constructions required in their proofs.
In Section~\ref{s:tori} we deal with the particular case of self-maps on wedge sums of tori. 
In Section~\ref{s:examples} we give a list of examples that illustrate the theory developed throughout the article and some of the challenges present in the generalization  of the context expounded here, which are summarized in the next section. 
In  Section~\ref{s:questions} we provide a list of open problems  and conjectures that have arisen   in the course of our research, and we make  some general comments. 

\section{Preliminaries and statements of the main results}\label{s:2}

A {\em pointed topological space} is a pair $(X,x_0)$ where $X$ is a topological space and $x_0\in X$. 
The {\em  wedge sum} of two pointed topological spaces $(X,x_0)$ and $(Y,y_0)$ is the quotient of the 
disjoint union $X\sqcup Y$ by identifying $x_0$ and $y_0$ to a single point, called the \emph{base point} of the wedge sum.
Given $s$ pointed topological spaces $X_1, X_2, \ldots, X_s$, the wedge sum 
$\bigvee_{j=1}^s  X_j =  X_1\vee X_2\vee \cdots \vee X_s$ 
is defined similarly.   


A wedge sum of circles is a graph. In the particular case of two circles, this space is the well-known figure ``8". A wedge sum of two-dimensional spheres is called a bouquet of spheres.


Let the maps:
 $\iota_j: X_j\to \bigvee_{i=1}^s  X_i$ be the natural inclusion and 
 $\nu_j: \bigvee_{i=1}^s  X_i\to X_j$ be   projection, where $1\leq j\leq s$.
Given a continuous map  $f\colon \bigvee_{i=1}^s  X_i \to \bigvee_{i=1}^s  X_i $, it induces continuous maps: $f_j:\bigvee_{i=1}^sX_i \to X_j$ 
by $f_j=\nu_j\circ f$ and $f_{jk}:X_j\to X_k$ so that: $f_{jk}:=\nu_k\circ f\circ\iota_j$ and 
$f_j=f_{j1}\vee\cdots\vee f_{js}$~(for the  definition of  the wedge sum of maps see~\cite[p. 39]{spanier}).
The maps $f_{jk}$ are called the \emph{ coordinates (or coordinate-maps) of $f$.}

Note that $f(x_0)=x_0$ since $f$ is a continuous pointed map, where $x_0$ is the base point of $\bigvee_{i=1}^sX_i$.

\smallskip

We say that the  map  $f$ is a \emph{permutative map on $\bigvee_{i=1}^sX_i$} if there exists 
a permutation $\sigma$ on $s$ elements so that the coordinates of $f$ are given by 
$$
f_{jk}(x)=x_0, \mbox{ for } x\in X_j,\, \mbox{ if }  k\neq \sigma(j).
$$
If the permutation $\sigma$ is the identity we say that  $f$ is a \emph{diagonal map}.
Notice that in the diagonal case, the map can be written as
$f=f_{11}\vee\cdots\vee f_{ss}$. 

\smallskip

Let $f$ be a permutative self-map on $\bigvee_{j=1}^sX_j$   with permutation $\sigma$ and  $\sigma=\pi_1\cdots\pi_q$ be  its decomposition  into cycles.
Let $\Lambda_1,\ldots,\Lambda_q$ be a partition of the index set $\{1,\ldots, s\}$, such that the cardinality of $\Lambda_j$ is $s_j$, and $s_1+\cdots+s_q=s$,
 so that the cycle $\pi_j$ acts on the $\Lambda_j$ and 
$f|_{\Lambda_j}:\bigvee_{\ell \in \Lambda_j}  X_\ell\to \bigvee_{\ell \in \Lambda_j}  X_\ell$, is a permutative map with permutation $\pi_j$.
Hence 
$f=f_{\Lambda_1}\vee\cdots\vee f_{\Lambda_q}$ is the  decomposition detailed in Lemma~\ref{lemma:decomposition} below.

Let us consider the example a permutative map $f$ on  $\bigvee_{i=1}^3X_i$, with $\sigma(1)=2$, $\sigma(2)=1$, and $\sigma(3)=3$.
We set  $\Lambda_1=\{1,2\}$, and $\Lambda_2=\{3\}$, so
$f=f_{\Lambda_1}\vee f_{\Lambda_2}$, where $f_{\Lambda_1}$ is a permutative map associated to the cycle $\pi_1(1)=2$, $\pi_1(2)=1$ and $f_{\Lambda_2}$ is a diagonal map on $X_3$.

\smallskip

We shall require an additional  condition on the spaces $X_i$ for this class of maps.
We say that a  permutative self-map on $\bigvee_{i=1}^sX_i$ is \emph{squared by blocks} if the spaces $X_j$, for $j\in\Lambda_i$, $1\leq i\leq q$, are pairwise homeomorphic.
In the particular case that a  permutative map  associated with a cyclic permutation is squared by blocks, all the spaces $X_j$, for $1\leq j\leq s$ are homeomorphic.

\medskip

In the following lines we introduce  the Lefschetz numbers and the Lefschetz zeta function.
Let $Y$ be a compact Euclidean neighborhood retract (ENR for short)  of dimension $N$~(\emph{cf.}~\cite[p. 527]{hatcher}).
We recall  that compact polyhedra and  CW complexes are ENR-s, moreover the class of ENR is closed under wedge sums.
 Let $f:Y\to Y$  be a continuous map, and 
$f_{*k}:H_k(Y,\Q)\to H_k(Y,\Q)$ the induced maps on the $k$-rational homology group of $Y$, with $0\leq k\leq N$.
The \emph{Lefschetz number} $L(f)$
is defined as
\begin{equation*}\label{number}
L(f):= \sum\limits_{k=0}^{N}(-1)^{k} \text{trace}(f_{*k}).
\end{equation*}

The \emph{Lefschetz zeta function} of $f$ is the exponential generating function of the Lefschetz numbers of the iterates of $f$, i.e.
$$
\zf(t):=\exp\left(\sum_{m=1}^{\infty}  \frac{L(f^m)}{m}t^m\right).
$$
A classic result  connecting  algebraic topology with
fixed point theory is the {\it Lefschetz Fixed Point Theorem} which
establishes the existence of a fixed point if $L(f)\neq 0$~(\emph{cf.}~\cite{lefschetz,brown}). As it is well-known the converse of this theorem is not true, consider for example the identity map on the torus.

\smallskip

Theorem~\ref{thm:main1} reduces the Lefschetz numbers and  zeta function of the map $f$ in terms of the Lefschetz numbers and zeta functions of its coordinate maps.

\begin{theorem}\label{thm:main1}
Let $f\colon \bigvee_{i=1}^sX_i \to \bigvee_{i=1}^sX_i$ be a continuous permutative  squared by blocks self-map with 
associated permutation $\sigma$. Let $s=s_1+s_2+\cdots+s_q$ be the partition
associated with the cyclic representation of $\sigma$. Then, for any partition 
$\Lambda_1,\Lambda_2, \ldots,\Lambda_q$  of the set $\{1,2,\ldots,s\}$ for which
$\#\Lambda_j = s_j$ for each $j=1,2,\ldots,q$, we have that
\begin{enumerate}
\item[(a)] The Lefcshetz numbers are: 
$$
L(f^m)=\left(\sum_{l=1}^t\mathbbm{1}_{s_l}(m)\left(\sum_{i\in\Lambda_l}L(f^m_{ii})-1\right)\right)+1,
$$
where $\mathbbm{1}_{s_l}$ is the characteristic  function on the integer multiples of $s_l$, i.e.
$$
\mathbbm{1}_{s_l}(m)= 
\left\{\begin{array}{ll}
1 & \text{ if } m\equiv 0 \pmod {s_l} \\
0 & \text{ if } m\not\equiv 0 \pmod {s_l}.
\end{array}\right.
$$

\item[(b)] The Lefschetz zeta function is given by the formula
\begin{equation}\label{eqn:zf-general}
\zeta_f(t) = \dfrac{1}{(1-t)} \prod_{j=1}^q \left(\prod_{i\in \Lambda_j}\left(1-t^{s_j}\right) \zeta_{f_{ii}^{s_j}}(t^{s_j})\right)^{1/{s_j}}.
\end{equation}
\end{enumerate}
\end{theorem}

In Examples 4 and 5 of Section~\ref{s:examples}, we show the importance of the hypothesis considered in Theorem~\ref{thm:main1}.

\medskip

We introduce the  \emph{Dold coefficients of} $f$, $\ell(f^m)$,
defined as
 \begin{equation}\label{eqn:l-chica}
\ell(f^m):=\sum_{r|m}\mu(r) L(f^{m/r})=\sum_{r|m}\mu({m/r}) L(f^r),
\end{equation}
where $\mu$ is the classical \emph{M\"obius function}, i.e.
 $$
\mu(m):=\left\{\begin{array}{ll}
               1 & \mbox{ if } m=1,\\
               0 & \mbox{ if } k^2|m \mbox{ for some } k\in\N, \\
              (-1)^r & \mbox{ if } m=p_1\cdots p_r  \mbox{ has distinct  primes factors.}
              \end{array}
       \right.
$$
By the M\"obius inversion formula~(\emph{cf.}~\cite{apostol}  Theorem 2.9):
$$
L(f^m)=\sum_{r|m} \ell(f^r).
$$
The numbers $\ell(f^m)$  were introduced by Dold in~(\cite{dold}); they are also called in  the literature 
the \emph{Lefschetz number of period $m$}~(\cite{Llibre-93}) 
  or
the \emph{algebraic multiplicity of order} $m$~(\cite{babo}) of the map $f$.
This allow us to define the \emph{set of  algebraic periods of} $f$  as
$$
\mbox{APer}_L(f):=\left\{m\,|\, \ell(f^m)\neq 0\right\}.
$$
We  remark  the fact $\ell(f^m)\neq 0$ does not necessarily implies the existence of a periodic point of period $m$ for $f$. 
However, this holds in some particular situations. For instance, whenever the space $Y$ is a 
differentiable manifold and $f$ is a Morse-Smale diffeomorphism~(see~\cite{Ll-Si:survey}, for a survey on the subject) or a $C^1$  transversal map~(\emph{cf.}~\cite{Llibre-93}).
The  terminology of the set of algebraic periods is mentioned explicitly in ~\cite{GMM}. 
The  intersection of  the set of  algebraic periods and the odd non-negative integers have  been widely   studied in the context  of quasi--unipotent maps (in this  case this intersection set  is called the minimal set of Lefschetz periods) during the last years, see  for example~\cite{Ll-Si:survey} and  references therein. 

Using the numbers $\ell(f^m)$ allow us to express formally the Lefschetz zeta function as an infinite Euler product~(\emph{cf.}~\cite{dold}):
\begin{equation}\label{eqn:lzf-euler}
\zf(t)=\prod_{m\geq 1}(1-t^m)^{-\dfrac{\ell(f^m)}{m}}.
\end{equation}

Theorem~\ref{thm:l-chica} gives reduction formulae of  the Dold coefficients of the map  $f$ in terms of its coordinate maps.  Corollary~\ref{coro:aperL}, yields information about the set of algebraic periods of $f$.

\begin{theorem}\label{thm:l-chica}
 Let $f\colon \bigvee_{i=1}^sX_i \to \bigvee_{i=1}^sX_i$ be a continuous permutative squared by blocks self-map with 
associated permutation $\sigma$. Let $s=s_1+s_2+\cdots+s_q$ be the partition
associated with the cyclic representation of $\sigma$. Then, for any partition 
$\Lambda_1,\Lambda_2, \ldots,\Lambda_q$  of the set $\{1,2,\ldots,s\}$ for which
$\#\Lambda_j = s_j$ for each $j=1,2,\ldots,q$, we have for $m>1$,
$$
\ell(f^m)=\sum_{j=1}^q \mathbbm{1}_{s_j}(m) \left( \sum_{i\in\Lambda_j} \ell(f_{ii}^m)\right),
$$
where $\mathbbm{1}_{s_j}$ is the characteristic  function on the integer multiples of $s_j$,
and $\ell(f)=L(f)$.
\end{theorem}

A straightforward consequence of this theorem is 
the following corollary: 
 \begin{corollary}\label{coro:aperL}
 Let $f$ be a continuous permutative map on $\bigvee_{j=1}^q X_j$ as in Theorem~\ref{thm:l-chica}.
 \begin{enumerate}
\item[(a)] If $s_j>1$ and $m\not\equiv 0 \pmod {s_j}$,  for $1\leq j\leq q$ then $\ell(f^m)= 0$.
\item[(b)] If for a given $m$ there exists a non zero $\ell(f_{ii}^m)$, for some $i$ and  all non-zero $\ell(f_{ii}^m)$ for $i\in\Lambda_j$ and $1\leq j\leq q$    have the same sign; then $m\in \mbox{APer}_L(f)$.
\item[(c)]
$$
 \mbox{APer}_L(f) \displaystyle\subset \bigcup_{j=1}^q\bigcup_{i\in\Lambda_j} \mbox{APer}_L(f_{ii}).$$
 \end{enumerate}
 \end{corollary}

An interesting application of Theorem~\ref{thm:main1} in the case of self-maps on  wedge sums of tori is Theorem~\ref{thm: torus-lppf}.
We recall that a continuous  map $f:X\to X$ is \emph{Lefschetz periodic point free} if $L(f^m)=0$, for all $m\geq  1$, i.e. $\zf(t)\equiv  1$.

  \begin{theorem}\label{thm: torus-lppf}
 Let $f:\bigvee_{i=1}^s \T^n\to \bigvee_{i=1}^s \T^n$ be as  a continuous permutative map with non-zero eigenvalues then $f$ is not Lefschetz periodic point free.
 \end{theorem}
 
\section{Proofs of the main results and applications}\label{s:proofs}

Throughout the article we shall asume that the spaces $X_j$ are path-connected. 
Given a continuous map $f\colon \bigvee_{j=1}^s  X_j \to \bigvee_{j=1}^s  X_j $ we shall see in the next lemma that it is possible 
to find a decomposition of the sum $\bigvee_{j=1}^s  X_j $ into subsums of the form
$\bigvee_{j\in \Lambda}  X_j$ which are invariant under the action of $f$, 
so that each summand is indecomposable in the sense that no nontrivial subsum is invariant under $f$.
We recall that a set $Z\subset X$ is an invariant set for the map 
$g\colon X\to X$, if $g(Z)\subset Z$.

\begin{lemma}\label{lemma:decomposition}
For any self-map $f\colon \bigvee_{j=1}^s  X_j \to \bigvee_{j=1}^s  X_j $ there exists a partition  
$\Lambda_1, \Lambda_2,\ldots,\Lambda_q$ of $\{1,2,\ldots,s\}$ such that
each $\bigvee_{j\in\Lambda_k}  X_j$ is an $f$-invariant indecomposable subsum
of $\bigvee_{j=1}^s  X_j$.  
\end{lemma}
\begin{proof} Let us consider the graph $\mathcal{G}_f = (\mathcal{V}_f,\mathcal{E}_f)$ having 
as a set of vertices $\mathcal{V}_f = \{1,2,\ldots, s\}$ and edges defined by the condition
$\{i,j\}\in \mathcal{E}_f$ if and only if there exist $x_i\in X_i$ and $ x_j\in X_j$ with $x_i,x_j \not = x_0$ (here $x_0$ is the base point)  such that 
$f(x_i) = x_j$ or $f(x_j) = x_i$. 
An equivalent relation  is defined by $i \sim_f j$ if and only if
$i=j$ or there is a path in $\mathcal{G}_f $ joining $i$ and $j$. The equivalent classes form the required
partition $\Lambda_1, \Lambda_2,\ldots,\Lambda_q$ of $\{1,2,\ldots,s\}$. 
Indeed: the fact  $i\in \Lambda_k$ and 
$f(X_i)\cap X_j \not = \emptyset$ implies $j\in \Lambda_k$,    yields $f(\bigvee_{\ell \in \Lambda_k}  X_\ell)\subset \bigvee_{\ell \in \Lambda_k}  X_\ell$.
Similarly
$f(\bigvee_{\ell \not\in \Lambda_k}  X_\ell)\subset \bigvee_{\ell \not\in \Lambda_k}  X_\ell$,
the indecomposibility is a consequence of the connectness of the classes of the graph  $\mathcal{G}_f$.      
\end{proof}

From Lemma~\ref{lemma:decomposition} follows that
$f=f_{\Lambda_1}\vee\cdots\vee f_{\Lambda_q}$, where $f_{\Lambda_j}$ is the restriction of $f$ to the invariant indecomposable sum 
$\bigvee_{\ell \in \Lambda_j}  X_\ell$. Notice that this decomposition is unique up relabeling or reordering the index sets $\Lambda_j$.

We consider the pointed spaces $X_i$  compact topological manifolds of dimension $n_i$ (it suffices to be 
compact polyhedra), the homology groups are: $H_k(\bigvee_{i=1}^sX_i,\Q)=\oplus_{i=1}^s H_k(X_i,\Q)$, 
for $1\leq k\leq N$, where $N=n_1+\cdots+n_s$~(\emph{cf.}~\cite[page 126]{hatcher}).

Let $f:\bigvee_{i=1}^sX_i \to \bigvee_{i=1}^sX_i $ be a continuous self-map and $f_{*k}: H_k(\bigvee_{i=1}^sX_i ,\Q)\to H_k(\bigvee_{i=1}^sX_i ,\Q)$ 
be the induced maps in the homology groups, with $0\leq k\leq N$. It has the following structure

\[
f_{*k}=\left(\,\begin{array}{c|c|c|c}
f_{{11}_{*k}} & f_{{12}_{*k}}  &\cdots & f_{{1s}_{*k}} \\
\hline
f_{{21}_{*k}} & f_{{22}_{*k}} &\cdots &f_{{2s}_{*k}}  \\
\hline
\vdots & \vdots &\vdots & \vdots\\
\hline
f_{{s1}_{*k}}& f_{{s2}_{*k}} &  \cdots &f_{{ss}_{*k}}
\end{array}\,\right),
\]
which is a $R\times  R$ matrix , where $R=r_1+\cdots+r_s$, and $r_i$  is the dimension of $H_k(X_i,\Q)$.
Note that $f_{ij_{*k}}:H_k(X_i,\Q)\to H_k(X_j,\Q)$ is the induced map in the $k$-th homology group  
of the coordinate map $f_{ij}:X_i\to X_j$. Throughout the text we use $(f_{ij})_{*k}$ and $f_{ij_{*k}}$ 
interchangeably.

Let us consider the  example of the permutation map described before with $s=3$:
\[
f_{*k}=\left(\,\begin{array}{c|c|c}
O_{{11}_{*k}} & f_{{2\sigma(2)}_{*k}}&  O_{{13}_{*k}} \\
\hline
 f_{{1\sigma(1)}_{*k}} &O_{{22}_{*k}}  &O_{{23}_{*k}}  \\
\hline
O_{{31}_{*k}}&  O_{{32}_{*k}}  &  f_{{3\sigma(3)}_{*k}}
\end{array}\,\right),
\]
where $O_{{ij}_{*k}}:H_k(X_i,\Q)\to H_k(X_j,\Q)$ is the null linear map, and $\sigma(1)=2$, $\sigma(2)=1$, $\sigma(3)=3$.
In the case that $f$ is squared by blocks, we have 
$X_1$ is homeomorphic to $X_2$.

Notice that the coordinates of the iterates of a self-map $f$ on $\bigvee_{i=1}^sX_i$ are obtained as follows. 
Consider the coordinates of $f^2$:
$f^2_{ij}=\vee_{l=1}^sf_{il}\circ f_{lj}$.
Hence the induced map in homology are:
$(f_{ij}^2)_{*k}=\sum_{l=1}^s (f_{il})_{*k}(f_{lj})_{*k}$. Inductively, we define $f^m_{ij}$ and 
$(f_{ij}^m)_{*k}$, for any positive integer $m$. Hence
\[
(f^m)_{*k}=\left(\,\begin{array}{c|c|c|c}
(f^m_{11})_{*k} & (f^m_{12})_{*k}   &\cdots &(f^m_{1s})_{*k}  \\
\hline
(f^m_{21})_{*k} &(f^m_{21})_{*k}&\cdots &(f^m_{2s})_{*k}  \\
\hline
\vdots & \vdots &\vdots & \vdots\\
\hline
(f^m_{s1})_{*k}& (f^m_{s2})_{*k} &  \cdots &(f^m_{ss})_{*k}
\end{array}\,\right).
\]

\medskip

In Theorem~\ref{thm:main1} we express the Lefschetz numbers and the Lefschetz zeta function of a squared by blocks permutative map on $\bigvee_{j=1}^s X_j$ in terms of its coordinate maps.
The proof is divided in some Lemmas, that state the result first for diagonal maps, later for cyclic permutative maps,  and finally we reduce the general case to these cases.

\medskip
 \begin{lemma}\label{lemma:L-diagonal}
  Let $f:\bigvee_{i=1}^sX_i \to \bigvee_{i=1}^sX_i$ be a diagonal map. Then 
 \begin{enumerate}
 \item[(a)] The Lefschetz numbers of $f$ are given by  $L(f^m)=\sum_{j=1}^s L(f_{jj}^m)-(s-1).$
 \item[(b)] The Lefschetz zeta function of $f$ is 
 \begin{equation}\label{eqn:zf-diagonal}
 \zf(t)=\left(1-t\right)^{s-1}\prod_{i=1}^s\zeta_{f_{ii}}(t).
 \end{equation}
 \end{enumerate}
 \end{lemma}
 
 \begin{proof}
 Since the maps $f_{*k}$ are diagonal by blocks, being the blocks $(f_{jj})_{*k}$, for $1\leq j\leq s$, the $\trace{(f_{*k}^m)}=\sum_{i=1}^s \trace({f_{ii_{*k}}^m)}$.
 
 Let $N=n_1+\cdots+n_s$, the dimension of $\bigvee_{i=1}^sX_i$.
 By the definition of the Lefschetz numbers:
 \begin{eqnarray*}
 L(f^m)&=&\sum_{k=0}^{N} (-1)^k \trace{(f^m_{*k})}
 =1+\sum_{k=1}^{N} (-1)^k \trace{(f^m_{*k})}\\
& = &1+\sum_{k=1}^{N} (-1)^k \left(\sum_{i=1}^s \trace{(f_{ii_{*k}}^m)}\right)
=1+\sum_{i=1}^s \left( \sum_{k=1}^{N}(-1)^k  \trace{(f_{ii_{*k}}^m)}\right)\\
&=&\sum_{i=1}^s \left(  \sum_{k=0}^{N}(-1)^k \trace{(f_{ii_{*k}}^m)}\right)-s+1
= \sum_{i=1}^sL(f_{ii}^m)-s+1.
 \end{eqnarray*}
 
 \smallskip
 
 By the definition  of the Lefschetz zeta function:
 \begin{eqnarray*}
 \zf(t)&=&\exp\left(\sum_{m\geq 0} L(f^m)\dfrac{t^m }{m}\right)\\
 &=&\exp\left(\sum_{m\geq 0} \left(\sum_{i=1}^s L(f_{ii}^m)-(s-1)\right)\dfrac{t^m }{m}\right)\\
 &=&\left(\prod_{i=1}^s\exp\left( \sum_{m\geq 0} L(f_{ii}^m)\dfrac{t^m )}{m}\right)\right)
\left( \exp\left(\sum_{m\geq 0}\dfrac{t^m }{m}\right)\right)^{-(s-1)}\\
  &=&\left(\prod_{i=1}^s \zeta_{f_{ii}}(t)\right)\left(1-t \right)^{s-1}.
 \end{eqnarray*}
 \end{proof}
 
 We would like to point out that the statement (a) of Lemma~\ref{lemma:L-diagonal} generalizes  Theorem 3.2.1 of~\cite{ferrario}, in the context of wedge sums of spaces. 
 
 \begin{lemma}\label{lemma:L-cyclic} 
 Let $f:\bigvee_{i=1}^sX_i \to \bigvee_{i=1}^sX_i $ be a 
continuous permutative squared by blocks map, with a cyclic permutation. Then  its Lefschetz numbers 
and Lefschetz zeta function are given by:

 \begin{enumerate}
 \item[(a)] $L(f^{sm})=\sum_{j=1}^s L(f_{jj}^{sm})-(s-1).$
 \item[(b)] $L(f^m)=1$ if  $m\not\equiv 0 \pmod s$.
 \item[(c)] 
 \begin{equation}\label{eqn:zf-cyclic}
 \zf(t)=\dfrac{1-t^s}{1-t}\left(\prod_{i=1}^s \zeta_{f_{ii}^s}(t^s)\right)^{\frac{1}{s}}.
 \end{equation}
 \end{enumerate}
 \end{lemma}
 
 \begin{proof}
 If $f$ is a permutation map with a cyclic permutation $\sigma$ then $\sigma^s$ is the identity, 
therefore  $f^s$ is a  diagonal map. Hence statement (a) follows by Lemma~\ref{lemma:L-diagonal}.
 
 \smallskip

 Let  $\sigma$ be the permutation associated to the  map $f$ and $f_{ij_{*k}}^m$ (or $(f_{ij}^m)_{*k}$) be 
the  block of $f^m_{*k}$ in the  position $(i,j)$.
 If $\sigma$ is a cycle different from the identity then $\sigma^m(i)\neq i$ for all $1\leq i\leq s$, and $m$ not a multiple of $s$.
 The fact that $f$ is squared by blocks yields 
 $$
( f_{i\sigma^m(i)}^m)_{*k}=f_{i\sigma(i)_{*k}}f_{\sigma(i)\sigma^2(i)_{*k}}\cdots f_{\sigma^{m-1}(i)\sigma^m(i)_{*k}}, 
\quad ( f_{ij}^m)_{*k} \text{ is trivial,  if }j\neq\sigma^m(i).
 $$
Therefore the blocks $f^m_{ii}$  are trivial, for $m \not\equiv  0 \pmod s$
 and the trace of $f^m_{*k}$ is zero, for $k\neq 0$. Hence $L(f^m)=1$, when $m\not\equiv 0 \pmod s$. 
 Proving in this way statement (b).
 
 \smallskip
 
 By the definition of the Lefschetz zeta function we have
 \begin{eqnarray*}
 \zf(t)&=&\exp\left(\sum_{m\geq 1} L(f^m)\dfrac{t^m }{m}\right)\\
 &=&
 \exp\left(\sum_{m\geq 1} L(f^{sm})\dfrac{t^{sm} }{sm}\right)
 \exp\left(\sum_{m\not\equiv 0 \pmod s} L(f^m)\dfrac{t^m }{m}\right).
 \end{eqnarray*}

We consider:
\begin{eqnarray*}
\sum_{m\not\equiv 0 \pmod s} L(f^m)\dfrac{t^m }{m}&=&\sum_{m\not\equiv 0 \pmod s} \dfrac{t^m }{m}\\
&=& \sum_{m=1}^{\infty} \dfrac{t^m }{m} -\sum_{m\equiv 0 \pmod s} \dfrac{t^m }{m}\\
&=& -\log\left(1-t\right)-\sum_{m=1}^{\infty}\frac{t^{ms}}{ms}\\
&=& -\log\left(1-t\right)+\frac{1}{s}\log\left(1-t^s\right).
\end{eqnarray*}
Hence 
$$
\exp\left(\sum_{m\not\equiv 0 \pmod s} L(f^m)\dfrac{t^m }{m}\right)=
\left(1-t^s\right)^{1/s}\left(1-t\right)^{-1}.
$$

On the other hand:
$$
\sum_{m=1}^{\infty} L(f^{sm})\dfrac{t^{sm} }{sm}
=\dfrac{1}{s}\sum_{m=1}^{\infty} L((f^s)^{m})\dfrac{(t^s)^{m} }{m}=\frac{1}{s}\log\left(\zeta_{f^s}(t^s)\right).
$$
Since $f^s$ is a diagonal map, by Lemma~\ref{lemma:L-diagonal}, it follows
$$
\exp\left(\sum_{m\geq 0} L(f^{sm})\dfrac{t^{sm} }{sm}\right)=
((1-t^s)^{s-1}\prod_{i=1}^s \zeta_{f_{ii}^s}(t^s))^{1/s}.
$$
It yields:
\begin{eqnarray*}
\zeta_f(t)&=&((1-t^s)^{s-1}\prod_{i=1}^s \zeta_{f_{ii}^s}(t^s))^{1/s}\left(1-t^s\right)^{1/s}\left(1-t\right)^{-1}\\
&=& \left(\dfrac{1-t^s}{1-t}\right)\left(\prod_{i=1}^s\zeta_{f_{ii}^s}(t^s)\right)^{1/s}.
\end{eqnarray*}
 \end{proof}

 \begin{remark}
  The expression~(\ref{eqn:zf-cyclic}) is a rational function. In fact consider 
 $$
 (f_{ii}^m)_{*k}= f_{i\sigma(i)_{*k}}f_{\sigma(i)\sigma^2(i)_{*k}}\cdots f_{\sigma^{m-1}(i)\sigma^m(i)_{*k}}
=f_{i\sigma(i)_{*k}}f_{\sigma(i)\sigma^2(i)_{*k}}\cdots f_{{\sigma^{m-1}(i)i}_{*k}}.
 $$
 Notice that all $(f_{ii}^m)_{*k}$ have the same factors, but in different order. Therefore all the eigenvalues of 
$(f_{ii}^m)_{*k}$, are the same for $1\leq i\leq s$, hence $L(f_{ii}^{sm})=L(f_{jj}^{sm})$ and $\zeta_{f_{ii}^s}=\zeta_{f_{jj}^s}$.
 \end{remark}

Notice that the fact $L(f^m)=1$ for $m\not\equiv 0 \pmod s$, implies 
(by the Lefschetz fixed point Theorem) that $f^m$ has a fixed point. However the only 
fixed point (when $f$ is a permutative cyclic map)  is $x_0$, the ``joining point''  or base point of the wedge sum of the spaces. The map $f^m$, 
it does not have other fixed points (when $m$ is not a multiple of $s$) since $f^m$ permutes the different $X_i$'s.

\begin{proof}[Proof of Theorem~\ref{thm:main1}]
We have seen in  Lemma~\ref{lemma:decomposition} that the map $f$ can be decomposed as $f=f_{\Lambda_1}\vee\cdots\vee f_{\Lambda_q}$.
Using Lemma~\ref{lemma:L-diagonal}, we obtain 
$$
L(f^m)=\sum_{l=1}^s L(f^m_{\Lambda_l})-(s-1).
$$
Since each map $f_{\Lambda_l}$ is permutatively cyclic, Lemma~\ref{lemma:L-cyclic} yields:
$$
L(f^m_{\Lambda_l})=
\left\{\begin{array}{ll} 1+\sum_{i\in\Lambda_l} \left(L(f^m_{ii}) -1\right),  & m\not\equiv 0 \pmod {s_l};\\
& \\
1,  &  m\not\equiv 0 \pmod {s_l}.
\end{array}
\right.
$$ 
 Writing this expression using the characteristic  function $\mathbbm{1}_{s_l}$, we get the formula in the statement (a).

\medskip

As far the Lefschetz zeta function is concerned,
Lemma~\ref{lemma:L-diagonal} yields:
$$
\zf(t)=(1-t)^{q-1}\prod_{i=1}^q \zeta_{f_{\Lambda_i}}(t)=(1-t)^{-1} \prod_{i=1}^q\left( (1-t) \zeta_{f_{\Lambda_i}}(t)\right).
$$
On the other hand, by Lemma~\ref{lemma:L-cyclic} we obtain:
$$
\zeta_{f_{\Lambda_i}}(t)=\dfrac{1-t^{s_i}}{1-t}\left(\prod_{j=1}^{s_i} \zeta_{f_{jj}^{s_i}}(t^{s_i})\right)^{1/s_i}.
$$
Combining these expressions, we get
$$
\zf(t)=\dfrac{1}{1-t}\prod_{i=1}^q\left(\prod_{j\in\Lambda_i} (1-t^{s_i})\zeta_{f_{jj}}(t^{s_i})\right)^{(1/s_i)}.
$$
\end{proof}

For the purpose of providing a formula that can be used for effective computations 
we can consider particular examples of the sets  $\Lambda$'s. For example, 
if the orbits of $\sigma$ are segments of non-decreasing size, then 
they must begin at the points $1$, $1+s_1$, $1+s_1+s_2$, $\ldots, 1+s_1+\cdots+s_{q-1}$ 
and consequently have the form
\begin{eqnarray*}
\Lambda_1 & = & \{1, \ldots, s_1\}, \\
\Lambda_2 & = & \{1+s_1, \ldots, s_1+s_2\}, \\
 & \vdots & \\
\Lambda_q & = & \{1+s_1+\cdots+s_{q-1}, \ldots, s_1+\cdots+s_{q} = s\}, \\ 
\end{eqnarray*}   
obtaining this way the formula expressed entirely in terms of the values of the partition of~$s$.  
In this case the equation~(\ref{eqn:zf-general}) can be expressed as
$$
\zeta_f(t) = \dfrac{1}{1-t} \prod_{j=1}^q 
\left(\prod_{i=1+s_0+s_1+\cdots+s_{j-1}}^{s_0+s_1+\cdots+s_j}\left(1-t^{s_j}\right) \zeta_{f_{ii}^{s_j}}(t^{s_j})\right)^{1/{s_j}},
$$
 where $s_0=0$.
 
 \smallskip
 
 The Lefschetz zeta function of a continuous self-map $f:Y\to Y$, 
on any compact polyhedron $Y$, is a rational function. Moreover it 
satisfies the formula~(\emph{cf.}~\cite[pg. 38]{franks2}):
 \begin{equation}\label{eqn:zeta-product}
\zf(t)=\prod_{k=0}^N\det(Id_{*k}-t f_{*k})^{(-1)^{k+1}},
\end{equation}
where $N=\dim Y$, $m_k=\dim H_k(Y,\Q)$, $Id:=Id_{*k}$ is the
identity map on $H_k(Y,\Q)$, and by convention $\det(Id_{*k}-t
f_{*k})=1$ if $m_k=0$. 
The formula (\ref{eqn:zeta-product}) gives an alternative way to compute 
the Lefschetz zeta function, without using the Lefschetz numbers explicitly. 
Combining this identity with the formula~(\ref{eqn:zf-general}), we can express the Lefschetz zeta function in term of determinants of the blocks of $f_{*k}$. 

\begin{corollary}
Let $f$ be a continuous permutative squared by blocks self-map on  $\bigvee_{i=1}^sX_i$, as in the statement of Theorem~\ref{thm:main1}. Then
$$
\zeta_f(t) = \dfrac{1}{(1-t)} \prod_{j=1}^q \left(\prod_{i\in \Lambda_j}\left(1-t^{s_j}\right) \left(\prod_{k=0}^{n_i}\det(Id_{*k}-t^{s_j}f^{s_j}_{ii_{*k}})^{(-1)^{k+1}}\right)\right)^{1/{s_j}},
$$
where $n_i$ is the dimension of $X_i$.
\end{corollary}

\medskip

In Theorem~\ref{thm:l-chica}, we express the Dold coefficients of a squared by blocks permutative map on $\bigvee_{j=1}^s X_j$ in terms of its coordinate maps.
The strategy of the proof is similar to the one used in the proof of Theorem~\ref{thm:main1}, i.e.  we start with  the case of  diagonal maps (Lemma~\ref{lemma:ell-diagonal}), later we consider cyclic permutative maps (Lemma~\ref{lemma:ell-permutation}) and finally  these results are used for proving  the general case.

\begin{lemma}\label{lemma:ell-diagonal}
  Let $f:\bigvee_{j=1}^s X_j \to \bigvee_{j=1}^s X_j  $ be a diagonal map. Then 
  $\ell(f^m)=\sum_{i=1}^s \ell(f_{ii}^m)$, if $m>1$ and $\ell(f)=L(f)$.
   \end{lemma}
 \begin{proof}
 From~(\ref{eqn:l-chica}) and Lemma~\ref{lemma:L-diagonal} follows
 \begin{eqnarray*}
 \ell(f^m)&=&\sum_{r|m}\mu({m/r}) L(f^r)=\sum_{r|m}\mu({m/r}) \left(\sum_{i=1}^s L(f_{ii}^r)-(s-1)\right)\\
 &=&\left(\sum_{i=1}^s\sum_{r|m}\mu({m/r})L(f_{ii}^r)\right)-\left(\sum_{r|m}\mu(m/r)(s-1)\right)\\
  &=&\left(\sum_{i=1}^s\ell(f_{ii}^m)\right)-\left((s-1)\sum_{r|m}\mu(m/r)\right).
 \end{eqnarray*}
 Since $\sum_{r|m}\mu({m/r})=0$, if $m>1$ and $\sum_{r|m}\mu({m/r})=1$, if $m=1$~(\emph{cf.}~\cite{apostol} Theorem 2.1), it yields
 $$
 \ell(f^m)=\left\{\begin{array}{ll}
 \displaystyle\sum_{i=1}^s\ell(f^m_{ii}) & \mbox{ if } m>1\\
  & \\
 L(f) & \mbox{ if } m=1.
                \end{array}\right.
 $$
 \end{proof}
 
  \begin{lemma}\label{lemma:ell-permutation} 
  Let $f:\bigvee_{j=1}^s X_j  \to \bigvee_{j=1}^s X_j $ be a continuous permutative  squared  by blocks map, with a cyclic permutation. 
Then  its Lefschetz numbers of period $m$ are given by:
 \begin{enumerate}
 \item[(a)] $\ell(f^{sm})=\sum_{j=1}^s \ell(f_{jj}^{sm}).$
 \item[(b)] $\ell(f^m)=0$,  if  $m\not\equiv 0 \pmod s$ and $m>1$.
 \end{enumerate}
 
 \end{lemma}
\begin{proof} 
Part (a) is straightforward  from Lemma~\ref{lemma:L-cyclic}-(a).
For part (b), 
notice that if $r|m$ and $m\not\equiv 0 \pmod s$ then $r\not\equiv 0\pmod s$.   Lemma~\ref{lemma:L-cyclic}-(b), states $L(f^r)=1$, for $r\not\equiv 0\pmod s$.
Hence, for $m\not\equiv 0 \pmod s$, it yields
$$
\ell(f^m)=\sum_{r|m}\mu(m/r)L(f^r)=\sum_{r|m}\mu(m/r)=0.
$$
\end{proof}

\begin{proof}[Proof of Theorem~\ref{thm:l-chica}]
Lemmas~\ref{lemma:decomposition} and~\ref{lemma:ell-diagonal} yield
$$
\ell(f^m)=\sum_{r|m}\mu(m/r)L(f^r)=\sum_{r|m}\mu(m/r)\sum_{j=1}^q L(f_{\Lambda_j}^m)=\sum_{j=1}^q \sum_{r|m}\mu(m/r)L(f_{\Lambda_j}^m)=\sum_{j=1}^q \ell(f_{\Lambda_j}^m).
$$
By Lemma~\ref{lemma:ell-permutation}  follows:
$$
\ell(f_{\Lambda_j}^m)= \left\{\begin{array}{ll} 
\displaystyle\sum_{i\in\Lambda_j}\ell(f_{ii}^m),  & \mbox{ if } m\equiv 0 \pmod {s_j};\\
 & \\
0, & \mbox{ if } m\not\equiv 0 \pmod {s_j}.
\end{array}\right.
$$
Combining these expressions and using  characteristic functions, we get
$$
\ell(f^m)=\sum_{j=1}^q \ell(f^m_{\Lambda_j})= \sum_{j=1}^q \mathbbm{1}_{s_j}(m)\left(\sum_{i\in\Lambda_j} \ell(f^m_{ii})\right).
$$
\end{proof}

 
 \section{Self-maps on wedge sums of tori}\label{s:tori}
 
 In this section we shall consider continuous map on a wedge sum of $n$-dimensional tori, i.e. $\bigvee_{i=1}^s \T^n$, with $n>1$.
 As we pointed out in the Introduction, when $n=1$, we have 
 $\bigvee_{i=1}^s\SSS^1$ is a graph and graph maps have been extensively studied.
 
 Since the homology spaces of the $n$-dimensional torus form an exterior algebra, we have, $H_0(\bigvee_{i=1}^s \T^n,\Q)=\Q$ and for $1\leq k\leq n$: 
$$
H_k\left(\bigvee_{i=1}^s\T^n,\Q\right)=\underbrace{\Q\oplus\cdots\oplus\Q}_{sr}, \quad \text{ where } r={ n \choose k};
$$
i.e., the $k$-th Betti number of the space is $sr$, where $r$ is the combinatorial number of $n$ choose $k$.

 \begin{proof}[Proof of Theorem~\ref{thm: torus-lppf}]
 We shall show in this case 
that the  Lefschetz  zeta function of  $f$ is not the constant function $1$.

 Theorem~\ref{thm:main1} yields
  $$
 \zeta_f(t) = \dfrac{1}{(1-t)} \prod_{j=1}^q \left(\prod_{i\in \Lambda_j}\left(1-t^{s_j}\right) \zeta_{f_{ii}^{s_j}}(t^{s_j})\right)^{1/{s_j}}.
 $$
 
 Since  Lefschetz zeta functions are rational functions, we define its degree as the degree of the numerator minus  the degree of the denominator.
 For toral  maps with non zero eigenvalues, the degree of their Lefschetz zeta function is zero~\cite{BS}.
 Notice  that in this case the maps $f_{ii}$,  are toral  maps defined on $\T^n$.  
 Hence the degree of 
 $$
 \left(\prod_{i\in \Lambda_j}\left(1-t^{s_j}\right) \zeta_{f_{ii}^{s_j}}(t^{s_j})\right)^{1/{s_j}}
 $$
 is zero.
 Therefore the degree of $\zf(t)$ is $-1$, which implies that  $\zf(t)$ is not constant, i.e. the map $f$ is not Lefschetz periodic point free.
 \end{proof}

 An interesting consequence of Theorem~\ref{thm:l-chica} in the context of self-maps on wedge sum of tori is the following corollary:
 
 \begin{corollary}\label{coro:n-prime}
 Let $n$ be an odd prime. 
 Let $g_c:\T^n\to \T^n$ be a continuous map such that   the characteristic polynomial of the  induced map $(g_c)_{*1}$ is $t^n-c$, with $c>2$ a positive integer.  
 Let $f:\bigvee_{i=1}^s \T^n\to \bigvee_{i=1}^s \T^n$ be a permutative map with permutation $\sigma$ such that its coordinates are given by:
 $f_{j\sigma(j)}=g_{c_j}$ and $f_{jk}$ is the constant map equal to the base point for $k\neq \sigma(j)$, and $c_j>2$ a positive integer, for $1\leq q\leq s$.
 Then $\mbox{APer}_L(f)=\N$.
 \end{corollary}
 
 \begin{proof}
 In~\cite{BGS}, it was proved that $\ell(g_c^m)<0$, for all $m>0$, when $c>2$ and $n$ an odd prime. From Theorem~\ref{thm:l-chica} follows that
 the set of the algebraic integers of $f$ is the set of the positive integers.
 \end{proof}
 
 In~\cite{BGS},  it was conjectured that $\ell(g_c^m)<0$, for all $m>0$, when $c>2$ any positive integer $n>1$. Therefore we conjecture that  Corollary~\ref{coro:n-prime} holds true for all positive $n$.
 \section{Examples}\label{s:examples}
 
 In the following examples, we shall illustrate the ideas of the previous sections in some particular cases, and show that the permutative case is somehow the general setting. 
 We explore this situation in more detail in Section~\ref{s:questions}.
 
 \noindent
 {\bf Example 1}: 
Consider the self-map $f$ on $\T^2\vee\T^2$, such that 
the coordinate-maps $f_{1,1}$, $f_{1,2}$, $ f_{2,1}$, $ f_{2,2}$ are given by the toral maps whose matrix representation are
$$
f_{1,1}=f_{2,2}= \left(\begin{array}{cc} 0 & 0\\ 0& 0\end{array}\right),
\quad f_{1,2}= \left(\begin{array}{cc} 0 & 1\\ 1& 0\end{array}\right), \quad
f_{2,1}= \left(\begin{array}{cc} 0 & -1\\ -1& 0\end{array}\right).
$$ 
If $\gamma_1^1,\gamma_1^2,\gamma_2^1,\gamma_2^2$ are the generators of $H_1(\T^2\vee\T^2)$, then
$$
f_{*1}(\gamma_1^1 )=-\gamma_2^2, \quad 
f_{*1}(\gamma_1^2)= -\gamma_2^1, \quad 
f_{*1}(\gamma_2^1)= \gamma_1^2, \quad 
f_{*1}( \gamma_2^2)= \gamma_1^1.
$$
So
$f_{*1}$ is given by this matrix
\begin{equation}\label{eqn:ex1}
(f_{*1})=\left(
\begin{array}{cccc}
0 & 0 & 0 & 1 \\
0 & 0 & 1 & 0 \\
0 & -1 & 0 & 0\\
-1 & 0 &0 & 0
\end{array}
\right),
\end{equation}

Notice that $f$ is a permutative map, with permutation $\sigma(1)=2$, $\sigma(2)=1$.

The graded structure (exterior algebra) of $\T^2$ induces a graded structure in $\T^2\vee\T^2$, then
 the generators of $H_2(\T^2\vee\T^2,\Q)$ are $\{\gamma_1^1\wedge\gamma_1^2,\gamma_2^1\wedge\gamma_2^2\}$,
since $\{\gamma_1^1,\gamma_1^2\}$ and $\{\gamma_2^1,\gamma_2^2\}$ are the generators of $H_1(X_1,\Q)$ and $H_1(X_2,\Q)$, respectively.
Therefore
$$
f_{*2}=\left\{ \begin{array}{lll}
\gamma_1^1 \wedge\gamma_1^2&\to &-\gamma_2^2\wedge-\gamma_2^1=-\gamma_2^1\wedge\gamma_2^2,\\
                       \gamma_2^1\wedge\gamma_2^2&\to& \gamma_1^2\wedge\gamma_1^1=-\gamma_1^1\wedge\gamma_1^2.
                        \end{array}
\right.
$$
Hence, it is given by the matrix
$$
(f_{*2})=\left(
\begin{array}{cc}
0 & -1 \\
-1 & 0 
\end{array}
\right),
$$
whose characteristic polynomial is $t^2-1$.

Here the Lefschetz numbers are:  
$$
L(f^m)=1-\trace(f_{*1}^m)+\trace(f_{*2}^m)=
\left\{\begin{array}{ll}
-1, & m=0 \pmod 4 \\
1, & m=1,3 \pmod 4\\
7, &  m=2 \pmod 4.
\end{array}
\right.
$$

Its Lefschetz zeta function, expressed as an Euler product is:
$$
\zf(t)=\dfrac{(1-t^4)^2}{(1-t)(1-t^2)^3}
$$
and using~(\ref{eqn:zeta-product}) it is given by
$$
\zf(t)=\prod_{k=0}^2\det(Id_{*k}-t f_{*k})^{(-1)^{k+1}}
=\dfrac{\det(Id_{*1}-tf_{*1})}{(1-t)\det(Id_{*2}-tf_{*2})}
=\dfrac{(1+t^2)^2}{(1-t)(1-t^2)}.
$$

The Dold coefficients of $f$ are:
  $\ell(f)=1$, $\ell(f^2)=6$, $\ell(f^4)=-8$, and $\ell(f^m)=0$, for $m\neq1,2,4$, i.e.
$
\mbox{APer}_L(f)=\{1,2,4\}.
$


\noindent
{\bf  Example 2:} Consider the self-map $f$ on $\T^2\vee\T^2$, such that 
its coordinates are given by the toral maps:
$$
f_{1,1}=f_{2,2}= \left(\begin{array}{cc} 0 & 0\\ 0& 0\end{array}\right),
\quad f_{1,2}= \left(\begin{array}{cc} 1 & 0\\ 0& 1\end{array}\right),
\quad
 f_{1,2}= \left(\begin{array}{cc} -a & 0\\ 0& -a\end{array}\right),
$$ 
with $a\in\N$.
Hence
$f_{*1}$ is given by this matrix
$$
(f_{*1})=\left(
\begin{array}{cccc}
0 & 0 & 1 & 0 \\
0 & 0 & 0 & 1 \\
-a& 0 & 0 & 0\\
0 & -a &0 & 0
\end{array}
\right).
$$

If $\gamma_1^1,\gamma_1^2,\gamma_2^1,\gamma_2^2$ are the generators of  $H_1(\T^2\vee\T^2)$, then

$$
f_{*1}(\gamma_1^1 )=-a\gamma_2^1, \quad 
f_{*1}(\gamma_1^2)= -a\gamma_2^2, \quad 
f_{*1}(\gamma_2^1)= \gamma_1^1, \quad 
f_{*1}( \gamma_2^2)= \gamma_1^2.
$$
Therefore
$$
f_{*2}=\left\{ \begin{array}{lll}
\gamma_1^1 \wedge\gamma_1^2&\to &(-a\gamma_2^1)\wedge(-a\gamma_2^2)=a^2\gamma_2^1\wedge\gamma_2^2,\\
                       \gamma_2^1\wedge\gamma_2^2&\to& \gamma_1^1\wedge\gamma_1^2.                        \end{array}
\right.
$$
Hence
$$
(f_{*2})=
\left(
\begin{array}{cc}
0 &1 \\
a^2 & 0 
\end{array}
\right).
$$

Here the Lefschetz numbers are:
$$
L(f^m)=1-\trace(f_{*1}^m)+\trace(f_{*2}^m)=
\left\{\begin{array}{ll}
1-4a^{k+1}+2a^{4k}, & m=4k; \\
1, & m=4k+1 \text{ or } m=4k+3;\\
1+4a^{k+1}+2a^{4k}, &  m=4k+2.
\end{array}
\right.
$$

Its Lefschetz zeta function is 
$$
\zf(t)=\prod_{k=0}^2\det(Id_{*k}-t f_{*k})^{(-1)^{k+1}}
=\dfrac{\det(Id_{*1}-tf_{*1})}{(1-t)\det(Id_{*2}-tf_{*2})}
=\dfrac{(1+at^2)^2}{(1-t)(1+(at)^2)}.
$$
It can be easily checked that algebraic periods are all the even numbers.
\medskip

 \noindent
{\bf Example 3:}
Consider the self-map $f$ on $\T^2\vee\T^2$, such that 
the coordinates   $f_{1,1},\, f_{1,2}, \, f_{2,1}, \,  f_{2,2}$ are given by the toral maps with matrix representation:
$$
f_{1,1}= \left(\begin{array}{cc} 0 & 0\\ -1& -1\end{array}\right),
\quad f_{1,2}= \left(\begin{array}{cc} 1 & 0\\ -1& -1\end{array}\right), \quad
f_{2,1}= \left(\begin{array}{cc} 0 & 0\\ 0& 1\end{array}\right), \quad
f_{2,2}= \left(\begin{array}{cc} 0 & 1\\ 0& 0\end{array}\right).
$$ 

Let $\{\gamma_1^1,\gamma_1^2,\gamma_2^1,\gamma_2^2 \}$ be the generators of $H_1(\T^2\vee\T^2,\Q)$, by the definition of $f$ follows
$f_{*1}(\gamma_1^1)=-\gamma_1^2$, 
$f_{*1}(\gamma_1^2)=\gamma_2^2-\gamma_1^2$,
$f_{*1}(\gamma_2^1)=\gamma_1^1-\gamma_1^2$ and
$f_{*1}(\gamma_2^2)=\gamma_2^1-\gamma_1^2$,  i.e. $f_{*1}$ is given by the matrix
\begin{equation}\label{eqn:ex3}
(f_{*1})=\left(
\begin{array}{cccc}
0 & 0 & 1 & 0 \\
-1 & -1 & -1 & -1 \\
0 & 0 & 0 & 1\\
0 & 1 &0 & 0
\end{array}
\right).
\end{equation}
Note that this map is not a permutation map on $\T^2\vee\T^2$, since the matrix of $f_{*1}$ is not a permutation matrix of blocks of size $2$.
 Since the generators of $H_2(\T^2\vee\T^2,\Q)$ are $\{\gamma_1^1\wedge\gamma_1^2,\gamma_2^1\wedge\gamma_2^2 \}$,
the map $f_{*2}$ expressed in terms of these generators:
 $$
 f_{*2}:\left\{ \begin{array}{ll}
 \gamma_1^1\wedge\gamma_1^2\to f_{*1}(\gamma_1^1)\wedge f_{*1}(\gamma_1^2)&=-\gamma_1^2\wedge(\gamma_2^2-\gamma_1^2)=-\gamma_1^2\wedge\gamma_2^2+\gamma_1^2\wedge\gamma_1^2 =-\gamma_1^2\wedge\gamma_2^2.\\

 \gamma_2^1\wedge\gamma_2^2\to  f_{*1}(\gamma_2^1)\wedge f_{*1}(\gamma_2^2)&=(\gamma_1^1-\gamma_1^2)\wedge(\gamma_2^1-\gamma_1^2)=\\
 & =\gamma_1^1\wedge\gamma_2^1 -\gamma_1^1\wedge\gamma_1^2-\gamma_1^2\wedge\gamma_2^1+\gamma_1^2\wedge\gamma_1^2=\\
 &=\gamma_1^1\wedge\gamma_2^1-\gamma_1^1\wedge\gamma_1^2-\gamma_1^2\wedge\gamma_2^1.\\
  \end{array}
 \right.
 $$

 Since the (co)homology of $\T^2\vee\T^2$ is  the direct sum (as rings) of the ring (co)homology of $\T^2$ and $\T^2$, we have $\gamma_i^t\wedge\gamma_j^r=0$, if $i\neq j$.
 Therefore $f_{*2}$ is the null map.
This is a contradiction since the coordinate maps: $f_{ij}:\T^2\to\T^2$, with $1\leq i,j\leq 2$, does not satisfy
$$
f_{*2}=\left(
\begin{array}{cc}
f_{11_{*2}} & f_{12_{*2}} \\
f_{21_{*2}} & f_{22_{*2}}
\end{array}
\right).
$$
For this reason, a continuous $f:\T^2\vee\T^2\to\T^2\vee\T^2$ with $f_{*1}$ as in~(\ref{eqn:ex3}) does not exists.

 
 \medskip
 
 \noindent
{\bf Example 4:} Let $f$ be a self-map on $\T^2\vee\T^2$ such that
its coordinates are:
$$
f_{1,1}= \left(\begin{array}{cc} 0 & 1\\ 0& 0\end{array}\right),
\quad f_{1,2}= \left(\begin{array}{cc} 0 & 0\\ 1& 0\end{array}\right), \quad
f_{2,1}= \left(\begin{array}{cc} 0 & 0\\ a& 0\end{array}\right), \quad
f_{2,2}= \left(\begin{array}{cc} 0 & 1\\ 0& 0\end{array}\right).
$$ 
with $a$ an integer. So
\begin{equation}\label{eqn:ex4}
(f_{*1})=\left(
\begin{array}{cccc}
0 & 1 & 0 & 0 \\
0 & 0 & 1 & 0 \\
0 & 0 & 0 & 1\\
a & 0 &0 & 0
\end{array}
\right).
\end{equation}

The matrix of $f_{*1}$ can not be written as a permutation matrix of blocks of  size 2.
As in  the Example 3, $f_{*2}$ is the null map. Therefore it does not exists 
such self-map $f$ on $\T^2\vee\T^2$, with such coordinates.  

Examples 3 and 4 show that not any $4\times 4$ integer matrix gives rise to a map $f:\T^2\vee\T^2\to\T^2\vee\T^2$ having such  matrix as $f_{*1}$.

 \noindent
{\bf Example 5:}
Let $\T^2\vee\T^2$, 
given by the coordinate-maps:
$$
f_{1,1}= \left(\begin{array}{cc} 0 & 1\\ -1& 0\end{array}\right),
\quad f_{1,2}= \left(\begin{array}{cc} 1 & 0\\ -1& -1\end{array}\right), \quad
f_{2,1}= f_{2,2}=\left(\begin{array}{cc} 0 & 0\\ 0& 0\end{array}\right).
$$

If $\{\gamma_1^1,\gamma_1^2,\gamma_2^1,\gamma_2^2 \}$ are the generators of $H_1(\T^2\vee\T^2,\Q)$ then 
$f_{*1}(\gamma_1^1)=-\gamma_1^2$, 
$f_{*1}(\gamma_1^2)=-\gamma_1^2$,
$f_{*1}(\gamma_2^1)=\gamma_1^1-\gamma_1^2$ and
$f_{*1}(\gamma_2^2)=-\gamma_1^2$,  i.e. $f_{*1}$ is given my the matrix
\begin{equation}\label{eqn:ex5}
(f_{*1})=\left(
\begin{array}{cccc}
0 & 1 & 1 & 0 \\
-1 & 0 & -1 & -1 \\
0 & 0 & 0 & 0\\
0 & 0 &0 & 0
\end{array}
\right).
\end{equation}
Note that it is not a permutative map, however the map $f_{*2}$ is well defined:
$$
(f_{*2})=\left(
\begin{array}{cc}
1 & -1 \\
0 & 0
\end{array}
\right).
$$
Therefore such continuous map $f:\T^2\vee\T^2\to\T^2\vee\T^2$ is well defined.

Here the Lefschetz numbers are:  
$$
L(f^m)=1-\trace(f_{*1}^m)+\trace(f_{*2}^m)=
\left\{\begin{array}{ll}
0, & m\equiv 0 \pmod 4 \\
2, & m\equiv 1,3 \pmod 4\\
4, &  m\equiv 2 \pmod 4;
\end{array}
\right.
$$
and its Lefschetz zeta function is:
$$
\zf(t)=\dfrac{1+t^2}{(1-t)^2}.
$$

Note that unlike the permutative case (Theorem \ref{thm:main1}), we have that $L(f^m)=0$ when $m\not\equiv 0 \pmod 2$.

Its Dold coefficients are $\ell(f)=\ell(f^2)=2$, $\ell(f^4)=-2$, and $\ell(f^m)=0$ if $m\neq 1,2,4$. Hence the set of algebraic periods is
$\mbox{APer}_L(f)=\{1,2,4\}.$
 
 \section{Open questions and further remarks}\label{s:questions}
 
 In this section we list some open problems and questions that arose from the topics dealt in the article, and we consider interesting and important to understand.

 \begin{enumerate}
 \item 
 Examples 3 and 4 show that not any $4\times 4$ integer matrix allows to define a map $f:\T^2\vee\T^2\to\T^2\vee\T^2$ having this  matrix as $f_{*1}$, more generally
  not any coordinate maps give a well-defined map on a wedge sum. This gives rise to the following  questions:
 \smallskip
 
 \begin{enumerate}
\item  Let $(X,x_0)$ be a pointed topological manifold, and $s\times s$ continuous maps $f_{ij}:X\to X$. Under which conditions the maps $f_{ij}$  produce a  continuous self-map $f$ on $\bigvee_{j=1}^s X$ having the maps $f_{ij}$ as coordinates?
 
 \smallskip
 
 Example 5 shows that there are continuous  and non-trivial self-maps on $\bigvee_{i=1}^s\T^n$, which are not of permutation type.
 
 \item In the particular case of $X=\T^n$: Which matrices of integer coefficients arise from self-maps on  wedge sums of the tori? 
 We point out, that for graph maps, i.e. $n=1$, there are no such obstructions, any possible integer matrix of the right dimension is allowed, since all the homology groups of $H_k(\bigvee_{i=1}^s\SSS^1, \Q)$ are trivial for $k>1$.

 \end{enumerate}
 \smallskip

\item The main difference between the study of the dynamics of  maps on  wedge sums of tori, of dimension larger than $1$ with graphs maps, is that for the latter there are no obstructions on the homology groups of dimension bigger than $1$. 
 
 \item Give general conditions for a map $f:\bigvee_{i=1}^s X_i \to \bigvee_{i=1}^s X_i$ to be  Lefschetz periodic point free and partially periodic point free. 
 Partial answers to this question have being given depending the nature of the spaces $X_i$. In the case of topological graphs and wedge sum of spheres it was studied partially before, for instance see~\cite{Ll,Ll-Si:2013,Si:2023}. 
 
\item Consider  a quasi-unipotent map $f:\bigvee_{i=1}^s \T^n \to \bigvee_{i=1}^s \T^n$, with permutation $\sigma$, whose non-constant coordinates $f_{i\sigma(j)}:\T^n\to\T^n$ are quasi-unipotent, i.e. the eigenvalues of the induced map on homology are roots of unity.
Let $C_{i}(t)=\Phi_{n(i)_1}(t)\cdots \Phi_{n(i)_r}(t)$ be the characteristic polynomial of $(f_{i\sigma(i)})_{*1}$, and $\Phi_k(t)$ the $k$-th cyclotomic polynomial.
Describe the set of periodic periods of $f$ in term of the $n(j)_j$, in a similar manner as it was done in~\cite{BGMS} for toral maps. 
In this case the problem cannot be easy done by using directly Theorem~\ref{thm:l-chica}, since could be cancelation in the corresponding Dold coefficients. 

\item Give a more precise description of the set of algebraic periods of permutative toral maps. 

 \end{enumerate}
 
\noindent 
\textbf{ Acknowledgements:}
 The authors would like to thank the anonymous referees for their
remarks which helped to improve the article.
 

\end{document}